\documentclass{amsart}

\usepackage{amsmath,amsxtra,amsthm,amssymb,amsfonts,graphicx} 

\theoremstyle{plain}

\newtheorem{thm}{Theorem}[section]

\newtheorem{defi}[thm]{Definition}

\newtheorem{lem}[thm]{Lemma}
\newtheorem{cor}[thm]{Corollary}
\newtheorem{alg}[thm]{Algorithm}

\theoremstyle{definition}

\newtheorem{rem}[thm]{Remark}
\newtheorem{exa}[thm]{Example}

\renewcommand{\phi}{\varphi}

\def\l{\left}
\def\r{\right}

\def\as{\!\mathrel{\mathop:}=} 

\def\tr{\operatorname{Tr}}

\def\fun{\col\hs\to\exrls}
\def\map{\col\hs\to\hs}
\def\geo{\col[0,1]\to\hs}

\def\nat{\mathbb{N}}
\def\nato{{\mathbb{N}_0}}
\def\rls{\mathbb{R}}
\def\exrls{(-\infty,\infty]}
\def\eps{\varepsilon}

\def\col{\colon}

\def\ol{\overline}

\def\argmin{\operatornamewithlimits{\arg\min}}

\def\cldom{\operatorname{\ol{dom}}}

\def\ts{\mathcal{T}}                          

\def\expe{\mathbb{E}}                         

\def\mi{\operatorname{Min}}

\def\cf{\mathcal{F}}

\def\hs{\mathcal{H}}   

\begin{document}
\title[Computing medians and means]{Computing medians and means in Hadamard spaces}
\author[M. Ba\v{c}\'ak]{Miroslav Ba\v{c}\'ak}
\date{\today}
\subjclass[2010]{Primary: 49M27; 62E99; Secondary: 51F99; 92B05}
\keywords{Hadamard space, mean, median, proximal point algorithm, the law of large numbers, tree space, computational phylogenetics, diffusion tensor imaging.}
\thanks{The research leading to these results has received funding from the
 European Research Council under the European Union's Seventh Framework
 Programme (FP7/2007-2013) / ERC grant agreement no 267087.}

\address{Miroslav Ba\v{c}\'ak, Max Planck Institute for Mathematics in the Sciences, Inselstr.~22, 04103 Leipzig, Germany}
\email{bacak@mis.mpg.de}

\begin{abstract}
The geometric median as well as the Fr\'echet mean of points in a Hadamard space are important in both theory and applications. Surprisingly, no algorithms for their computation are hitherto known. To address this issue, we use a~splitting version of the proximal point algorithm for minimizing a~sum of convex functions and prove that this algorithm produces a sequence converging to a minimizer of the objective function, which extends a recent result of D.~Bertsekas (2011) into Hadamard spaces. The method is quite robust and not only does it yield algorithms for the median and the mean, but it also applies to various other optimization problems. We moreover show that another algorithm for computing the Fr\'echet mean can be derived from the law of large numbers due to K.-T.~Sturm (2002).

In applications, computing medians and means is probably most needed in tree space, which is an instance of a Hadamard space, invented by Billera, Holmes, and Vogtmann (2001) as a tool for averaging phylogenetic trees. Since there now exists a polynomial-time algorithm for computing geodesics in tree space due to M.~Owen and S.~Provan (2011), we obtain efficient algorithms for computing medians and means of trees, which can be directly used in practice.
\end{abstract}

\maketitle
\section{Introduction}

Given positive weights $w_1,\dots,w_N$ satisfying $\sum w_n=1,$ and a finite set of points $a_1,\dots,a_N$ in a metric space $(X,d),$ we define its \emph{geometric median} as
\begin{equation} \label{eq:median}
 \Psi\l(\ol{w};\ol{a}\r)\as\argmin_{x\in X} \sum_{n=1}^N w_n d\l(x,a_n\r), 
\end{equation}
and its \emph{Fr\'echet mean} as
\begin{equation}\label{eq:mean}
 \Xi\l(\ol{w};\ol{a}\r)\as\argmin_{x\in X} \sum_{n=1}^N w_n d\l(x,a_n\r)^2,
\end{equation}
where we denote $\ol{w}\as\l(w_1,\dots,w_N\r)$ and $\ol{a}\as\l(a_1,\dots,a_N\r).$ These definitions will not function well in an arbitrary metric space, but as far as geodesic metric spaces of nonpositive curvature (so-called Hadamard spaces) are concerned, they become highly appropriate. Hadamard spaces include, apart from Hilbert spaces, Euclidean buildings, and some Riemannian manifolds (so-called Hadamard manifolds), also the Billera-Holmes-Vogtmann tree space (BHV tree space), which is a nonpositively curved cubical complex constructed in~\cite{billera} as a model space for phylogenetic trees. Computing means and medians in this setting can therefore be of importance in computational phylogenetics. Another area where our algorithm can be possibly used is \emph{diffusion tensor imaging.} We return to both of them shortly. Now we would like to mention one more application. Very recently, Fr\'echet means have also emerged in connection with so-called \emph{consensus algorithms} in Hadamard spaces~\cite{
grohs}.

Our goal in the present paper is to introduce algorithms for computing medians and means in Hadamard spaces. Research in this direction has already started. The first algorithms for computing means in Hadamard spaces appeared in~\cite[Section 4]{feragen}. The authors of~\cite{feragen} propose three different methods for computing the Fr\'echet mean in Hadamard spaces, but unfortunately, all of these methods fail to converge to a correct value even in a very simple situation, as will be demonstrated in Remark~\ref{rem:counterex} below. These three method are called Birkhoff's shortening, the centroid method and the weighted average method and the interested reader is referred to~\cite[Section 4]{feragen} for their respective definitions.
\begin{rem}\label{rem:counterex}
Let $(\hs,d)$ be a Hadamard space consisting of three geodesic rays issuing from the origin $0.$ This is an $\rls$-tree, and as a matter of fact the BHV tree space~$\ts_3.$ Consider three points $x,y,z\in\hs$ lying in distinct rays issuing from the origin~$0$ such that $d(0,z)=5,$ and $d(0,x)=d(0,y)=1.$ Then it is easy to see that the Fr\'echet mean $\Xi$ of $x,y,z$ with the uniform weights $\frac13,\frac13,\frac13$ lies on the geodesic $[0,z]$ and $d(0,\Xi)=1.$ On the other hand if we apply Birkhoff's shortening or the centroid method, we will get a point $c\in[0,z]$ such that $d(0,c)>\frac54.$ Finally, the weighted average method of the points $x,y,z$ yields a point $w\in[0,z]$ with $d(0,w)=\frac53,$ or the origin $0,$ depending on the order we choose.
\end{rem}
As we will observe in Section~\ref{sec:lln}, one way to approximate the Fr\'echet mean is an algorithm based on the law of large numbers. This observation was independently made also in~\cite{miller}. Our main result in this paper is the splitting proximal point algorithm which applies to computing both medians and means. As a matter of fact, this method can be used in a much broader class of optimization problems, which we describe later in this Introduction.

It is worth mentioning that in spite of an apparent similarity between~\eqref{eq:median} and~\eqref{eq:mean}, there is a substantial difference in the complexity of computing the median and the mean even in Euclidean spaces. While it is trivial to find a mean in finitely many steps, see~\eqref{eq:arithmetic} below, there exists no formula for computing a median in~$\rls^d,$ and we can use only approximation algorithms; see~\cite{bose} and the references therein. Another difference between the median and mean is that the former is not unique, that is, the set $\Psi\l(\ol{w};\ol{a}\r)$ contains more than one point in general, whereas $\Xi\l(\ol{w};\ol{a}\r)$ is always a singleton; see Theorem~\ref{thm:average}.

\subsection*{The BHV tree space and statistical biology} 
To increase the motivation and whet the appetite even more we will now take a closer look at the BHV space and its applications. The BHV tree space is a CAT(0) cubical complex whose elements are metric trees with a fixed number of terminal nodes and lengths assigned to all edges~\cite{billera}.

Metric trees with a fixed number of terminal vertices can represent evolutionary trees in phylogenetics. Then, given a finite collection of such trees, it is desirable to find an average tree. A natural candidate for this average is the Fr\'echet mean. On the other hand until now, no algorithm for its computation was available in the BHV tree space and therefore alternative concepts of an average were used in practice, for instance, a centroid~\cite{billera}, or a majority consensus~\cite{nye}. However, with our algorithms at hand, one can efficiently compute the Fr\'echet mean itself. Building upon these algorithms, a novel statistical model for phylogenetic inference was developed in~\cite{benner}. For a general mathematical background of contemporary phylogenetics, the reader is referred to~\cite{dress,pachter,semple}.

Apart from phylogenetics, tree-like structures emerge naturally in other subject fields of biology and computing an average tree is again of interest. Let us mention applications to the modeling of airway systems in human lungs~\cite{feragen,feragen2} and blood vessels~\cite{marron}.

All the algorithms presented in this paper require computing geodesics in the underlying Hadamard space. While it can be a difficult task in a general Hadamard space, there exists an efficient polynomial-time algorithm for computing geodesics in the BHV tree space due to M.~Owen and S.~Provan~\cite{owen,owenprovan}, which makes our algorithms directly applicable in practice.

\subsection*{Diffusion tensor imaging}

The space $P(n,\rls)$ of symmetric positive definite matrices $n\times n$ with real entries is a Hadamard manifold provided it is equipped with the Riemannian metric
\begin{equation*}
 \langle X,Y\rangle_A\as\tr\l(A^{-1}XA^{-1}Y \r),\qquad X,Y\in T_A\l(P(n,\rls)\r),
\end{equation*}
for every $A\in P(n,\rls);$ see~\cite[p. 314]{bh}. This manifold plays a key role in \emph{diffusion tensor imaging} as explained in~\cite{pennec,pennec2} and computing Fr\'echet means of a finite family of symmetric positive definite matrices is one of the crucial operations \cite[Section 3.7]{pennec}. Our algorithms can therefore find an application also in this area. On the other hand, it is unknown to the author whether there exists an efficient algorithm for computing geodesics in~$P(n,\rls).$


\subsection*{The proximal point algorithm and its applications}

Having explained the importance of medians and means in Hadamard spaces and the need for their computations, we will now introduce our tools.

Let $(\hs,d)$ be a Hadamard space. Our algorithms for computing medians and means are based on a~special version of the proximal point algorithm (PPA) for minimizing a convex function~$f$ of the form
\begin{equation} \label{eq:f}
f\as\sum_{n=1}^N f_n,
\end{equation}
where $f_n\fun$ are all convex and lower semicontinuous (lsc). The main trick here is that instead of applying iteratively the resolvent
\begin{equation*} J_\lambda(x)\as\argmin_{y\in \hs}\l[f(y)+\frac{1}{2\lambda}d(x,y)^2\r]\end{equation*}
of the function $f,$ we apply the resolvents
\begin{equation*} J_\lambda^n(x)\as\argmin_{y\in \hs}\l[f_n(y)+\frac{1}{2\lambda}d(x,y)^2\r]\end{equation*}
of its components $f_n$ either in cyclic or random order; see Definitions~\ref{def:cyclic} and~\ref{def:random} for the precise formulations of the algorithms. Such algorithms have been recently shown by D.~Bertsekas~\cite{bertsekas} to converge to a minimizer of $f$ when the underlying space is $\rls^d.$ We extend these results into locally compact Hadamard spaces and then apply them with $f_n$ being $f_n=w_n d\l(\cdot,a_n\r)$ in the case of the median~\eqref{eq:median}, and with $f_n= w_n d\l(\cdot,a_n\r)^2$ in the case of the mean~\eqref{eq:mean}. In either case, it is then easy to find explicit formulas for the resolvents $J_\lambda^n,$ and one hence obtains simple algorithms for computing medians and means, respectively. The detailed description is given in Section~\ref{sec:mm}. We note that Bertsekas' paper~\cite{bertsekas} presents a much more general approach to convex minimization problems using for instance various types of gradient methods, which however seem to be difficult to generalize into our setting. We also refer the interested reader to Bertsekas' paper for historical remarks about the use of splitting methods going back to Lions\&Mercier~\cite{lionsmercier} and Passty~\cite{passty}. In this connection, we also recommend~\cite{combettes}. The resolvents $J_\lambda$ of convex lsc functions in Hadamard spaces were first studied by J.~Jost~\cite{jost95} and U.~Mayer~\cite{mayer}.

While the present paper was under review, the following two papers appeared: S.~Ohta and M.~P\'alfia~\cite{ohtapalfia} extended the cyclic order version of the splitting PPA to other classes of geodesic spaces and S.~Banert~\cite{banert} studied the splitting PPA for two functions in (non locally compact) Hadamard spaces.

\subsection*{On the objective function}

The function~\eqref{eq:f} can apparently accommodate a variety of problems. For instance, we can put
\begin{equation} \label{eq:weighted}
 f(x)\as \sum_{n=1}^N w_n d\l(x,a_n\r)^p,\qquad x\in \hs,
\end{equation}
where $p\in[1,\infty).$ Then~$f$ is convex continuous and encompasses medians and means as special cases:
\begin{enumerate}
 \item If $p=1,$ then $f$ becomes the objective function in the \emph{Fermat-Weber problem} for optimal facility location and its minimizer is a median of the points $a_1,\dots,a_N$ with weights $w_1,\dots,w_N.$
 \item If $p=2,$ then a minimizer of~$f$ is the \emph{barycenter} of the probability measure
\begin{equation*} \pi\as\sum_{n=1}^N w_n \delta_{a_n},\end{equation*}
where $\delta_{a_n}$ stands for the Dirac measure at the point $a_n.$ In other words the mean of the points $a_1,\dots,a_N$ can be equivalently viewed as the barycenter of~$\pi.$ Barycenters of probability measures on Hadamard spaces were first studied by J.~Jost \cite{jost94}. For further details, the reader is referred to~\cite[Chapter 3]{jost2} and~\cite{sturm-conm} and~\cite[Chapter 2]{mybook}. 
\end{enumerate}

Another way of generalizing~\eqref{eq:median}, and \eqref{eq:mean} alike, is to replace the points $a_1,\dots,a_N$ by convex closed sets $C_1,\dots,C_N\subset \hs.$ Since the distance functions to such sets are convex continuous (see Example~\ref{exa:dist} below) in Hadamard spaces, the objective function in this problem is of the form~\eqref{eq:f}. Namely, we are to minimize the function
\begin{equation} \label{eq:boris}
 f(x) \as \sum_{n=1}^N w_n d\l(x;C_n\r),\qquad x\in \hs.
\end{equation}
Problems of this type have been recently studied from various perspectives~\cite{boris3,boris1,boris2}. Our approach, based on the proximal point algorithm, seems to be however novel even in linear spaces. An explicit algorithm is given in~\ref{rem:boris}.

We will now present several natural examples of convex lsc functions in Hadamard spaces to see that the proximal point algorithm is applicable in many more situations. Let still $(\hs,d)$ be a Hadamard space.
\begin{exa}[Indicator functions] \label{exa:indicator}
Let $C\subset \hs$ be a convex set. Define the~\emph{indicator function} of $C$ by
\begin{equation*} \iota_C(x)\as\l\{
\begin{array}{ll} 0, & \text{if } x\in C,  \\ \infty,  & \text{if } x\notin C. \end{array} \r. \end{equation*}
Then $\iota_C$ is a convex function and it is lsc if and only if $C$ is closed.
\end{exa}
The indicator function is often employed to convert a constrained minimization problem into an unconstrained one. Indeed, the minimization of the function~\eqref{eq:f} on a closed convex set $C\subset \hs$ is equivalent to the minimization of
\begin{equation*} \tilde{f}\as\iota_C+\sum_{n=1}^N f_n,\end{equation*}
over the whole space~$\hs.$
\begin{exa}[Distance functions] \label{exa:dist}
The function
\begin{equation}\label{eq:dist} x\mapsto d\l(x,x_0\r),  \quad x\in \hs,\end{equation}
where $x_0$ is a fixed point of $ \hs,$ is convex and continuous. The function $d\l(\cdot,x_0\r)^p$ for $p>1$ is strictly convex. More generally, the \emph{distance function} to a closed convex subset $C\subset \hs,$ defined as
\begin{equation*}d(x;C)\as\inf_{c\in C} d(x,c),  \quad x\in \hs,\end{equation*}
is convex and $1$-Lipschitz~\cite[p.178]{bh}.
\end{exa}
\begin{exa}[Displacement functions]
Let $T\map$ be an isometry. The \emph{displacement function} of $T$ is the function $\delta_T\col\hs\to[0,\infty)$ defined by
\begin{equation*}\delta_T(x)\as d(x,Tx),\quad x\in  \hs.\end{equation*}
It is convex and Lipschitz~\cite[p.229]{bh}.
\end{exa}
\begin{exa}[Busemann functions] \label{exa:busemann}
Let $c\col[0,\infty)\to  \hs$ be a geodesic ray. The function $b_c\col \hs\to\rls$ defined by
\begin{equation*} b_c(x)\as \lim_{t\to\infty} \l[d\l(x,c(t) \r) -t \r],  \quad x\in  \hs,\end{equation*}
is called the \emph{Busemann function} associated to the ray $c.$ Busemann functions are convex and $1$-Lipschitz. Concrete examples of Busemann functions are given in \cite[p.~273]{bh}. Another explicit example of a Busemann function in the Hadamard space of positive definite $n\times n$ matrices with real entries can be found in~\cite[Proposition~10.69]{bh}. The sublevel sets of Busemann functions are called \emph{horoballs} and carry a lot of information about the geometry of the space in question; see \cite{bh} and the references therein.
\end{exa}
\begin{exa}[Energy functional]
The energy functional is a convex lsc function on a Hadamard space of $\mathcal{L}^2$-mappings. It has been studied extensively at a varying level of generality \cite{fuglede2,fuglede1,gromov,jost95,jost97,ks}. Minimizers of the energy functional are called \emph{harmonic maps} and are of importance in both geometry and analysis. For a probabilistic approach to harmonic maps in Hadamard spaces, see~\cite{sturm-markov1,sturm-markov2,sturm-semigr}.
\end{exa}

We explicitly mention yet another application of the above version of the proximal point algorithm. Namely, if $C_1,\dots,C_N$ are closed convex subsets of $\hs$ such that
\begin{equation*}C_1\cap\dots\cap C_N\neq\emptyset,\end{equation*}
and if we set $f_n\as\iota_{C_n}$ in~\eqref{eq:f}, then it is easy to see that 
\begin{equation*}J_\lambda^n(x)=P_{C_n}(x),\end{equation*}
for every $x\in \hs,\;\lambda>0,$ and $n=1,\dots,N.$ Here $P_{C_n}$ stands for the metric projection onto $C_n;$ see Section~\ref{sec:preli}. The proximal point algorithm hence becomes the method of cyclic and random projections, respectively, and converges to a point $c\in\bigcap_{n=1}^N C_n.$ Such algorithms play an important role in optimization, for instance in convex feasibility problems; see~\cite{bauschketams,bauschke,baucom} and the references therein. If $N=2,$ both cyclic and random orders of the projections give the same approximating sequence (modulo repeating elements), and we get the so-called alternating projections, which were in Hadamard spaces studied in~\cite{apm}.

\subsection*{The Lie-Trotter-Kato formula}
There is also a tight connection to gradient flow semigroups for a function of the form~\eqref{eq:f}, since the proximal point algorithm is a~discrete time version of the gradient flow. 

We need the following notation. If $F\map$ is a mapping, we denote its $k$-th~power, with $k\in\nat,$ by
\begin{equation*}F^{(k)}x\as\l(F\circ\dots\circ F\r) x,\quad x\in \hs,\end{equation*}
where $F$ appears $k$-times on the right hand side.

Recall that the \emph{gradient flow semigroup} $\l(S_t\r)_{t\geq0}$ of $f$ is given as
\begin{equation} \label{eq:defsem}
S_t x\as\lim_{k\to\infty} \l(J_{\frac{t}k}\r)^{(k)} (x),\quad x\in \cldom f.
\end{equation}
The limit in \eqref{eq:defsem} is uniform with respect to $t$ on bounded subintervals of $[0,\infty),$ and $\l(S_t\r)_{t\geq0}$ is a strongly continuous semigroup of nonexpansive mappings; see \cite[Theorem~1.3.13]{jost-ch} and \cite[Theorem~1.13]{mayer}. Note however that formula~\eqref{eq:defsem} was in a nonlinear space used already in~\cite[Theorem 8.1]{reichshafrir}. In the same way we define the semigroups~$S_t^n$ of the components~$f_n,$ using the appropriate resolvents~$J_\lambda^n,$ for $n=1,\dots,N.$

The following nonlinear version of the Lie-Trotter-Kato formula was proved in~\cite{stoj}. It shows that, given a function $f$ of the form~\eqref{eq:f}, we can approximate the semigroup of $f$ by the resolvents of the components~$f_n.$
\begin{thm}[Stojkovic]
Let $(\hs,d)$ be a Hadamard space and $f\fun$ be of the form~\eqref{eq:f}. Then we have
\begin{equation} \label{eq:stoj}
S_t x =  \lim_{k\to\infty} \l(J_{\frac{t}{k}}^N\circ\dots\circ J_{\frac{t}{k}}^1\r)^{(k)}(x),
\end{equation}
for every $t\in[0,\infty)$ and $x\in\cldom f.$
\end{thm}
\begin{proof}
The proof given in~\cite{stoj} uses ultralimits of Hadamard spaces. A simpler proof relying on weak convergence appeared in~\cite{lie}. 
\end{proof}

\subsection*{The law of large numbers in Hadamard space}
The law of large numbers is related to the Fr\'echet mean in the classical linear setting as well as in Hadamard spaces. For a~historical remark, we refer the interested reader to~\cite[Remark 2.7a]{sturm}. As we shall see in the sequel, more precisely in Section~\ref{sec:lln}, the probabilistic point of view enables us to find an alternative algorithm for computing the Fr\'echet mean. We shall also compare the algorithms based on the PPA with this algorithm based on the law of large numbers in Section~\ref{sec:lln}.

It is worth mentioning that there exists a slightly different approach to barycenters as well as to the law of large numbers due to A.~Es-Sahib and H.~Heinich~\cite{sh} which is not equivalent to the approach mentioned above; see~\cite[Example~6.5]{sturm-conm}. For related ergodic theorems, we refer the interested reader to the recent papers~\cite{austin,navas}.

The author was informed that the algorithm relying upon the law of large numbers was independently discovered by E.~Miller, M.~Owen, and S.~Provan~\cite{miller}.
 
\subsection*{The organization of the paper} The following Section~\ref{sec:preli} is devoted to the rudiments of Hadamard space theory including a discussion on medians and means. The main results of the present paper are contained in Section~\ref{sec:ppa}. We prove that both the cyclic and random order versions of the PPA converge to a minimizer of the function in question. In Section~\ref{sec:mm} we apply the PPA to the case of medians and means, respectively, and obtain explicit and user-friendly algorithms for their computations. The last part, Section~\ref{sec:lln}, is devoted to an alternative algorithm for Fr\'echet means which relies upon the law of large numbers.

\subsection*{Acknowledgments} I would like to thank Megan Owen for bringing the question of computing medians in the BHV tree space to my attention and sharing her insight with me. I am also very grateful to Aasa Feragen, Ezra Miller, Tom Nye and Sean Skwerer for many inspiring discussions on this and related subjects during the \emph{Workshop on Geometry and Statistics in Bioimaging: Manifolds and Stratified spaces} in S{\o}nderborg, Denmark, in October 2012. Special thanks go to Philipp Benner, Martin Kell and Ezra Miller for their helpful comments on earlier versions of the manuscript. It is my pleasure to thank the referees for their valuable remarks and suggestions.

\section{Preliminaries} \label{sec:preli}
 
\subsection*{Hadamard spaces.}
We will now recall basic facts on Hadamard spaces. For further details on the subject, we refer the reader to~\cite{bh,jost2} or~\cite{mybook}. We adopt usual analysis/optimization notation. Positive and nonnegative integers are denoted by~$\nat$ and $\nato,$ respectively.

A metric space $(X,d)$ is called \emph{geodesic} if for each pair of points $x,y\in X$ there exists a geodesic which connects them. That is, there exists a mapping $\gamma\col[0,1]\to X$ such that $\gamma(0)=x,\gamma(1)=y,$ and
\begin{equation*}d\l(\gamma(s),\gamma(t)\r)=d(x,y)\:|s-t|,\end{equation*}
for $s,t\in[0,1].$ If for each point $z\in X,$ geodesic $\gamma\col[0,1]\to X,$ and $t\in[0,1],$ we have
\begin{equation} \label{eq:cat}
d\l(z,\gamma(t)\r)^2\leq (1-t) d\l(z,\gamma(0)\r)^2+td\l(z,\gamma(1)\r)^2-t(1-t) d\l(\gamma(0),\gamma(1)\r)^2,
\end{equation}
the space $(X,d)$ is called CAT(0). This property in particular implies that every two points are connected by a \emph{unique} geodesic. A complete CAT(0) space is called a \emph{Hadamard space.}

We will moreover assume the Hadamard spaces in our theorems be locally compact. Apart from the BHV tree space described in the Introduction, the class of locally compact Hadamard spaces includes Euclidean spaces, hyperbolic spaces, complete simply connected Riemannian manifolds of nonpositive sectional curvature (e.g. $P(n,\rls)$ mentioned above), Euclidean buildings, locally compact $\rls$-trees and CAT(0) complexes. The algorithm in Section~\ref{sec:lln} however works without the local compactness assumption.

We will now recall an inequality which goes back to the work of Reshetnyak. Its modern proof can be found in~\cite[Proposition 2.4]{sturm-conm}, or in~\cite[Lemma 2.1]{lang-gafa}.
\begin{lem} \label{lem:resh}
 Let $(\hs,d)$ be a Hadamard space. Then we have
\begin{equation*} d(x,y)^2 +d(u,v)^2\leq d(x,v)^2+d(y,u)^2 +2d(x,u)d(y,v),\end{equation*}
for any points $x,y,u,v\in\hs.$
\end{lem}

Given a pair of points $x,y\in \hs,$ we denote $(1-t)x+ty=\gamma(t),$ where $\gamma$ is the geodesic connecting $x$ and $y.$ We say that a set $C\subset \hs$ is \emph{convex} provided $x,y\in C$ implies $(1-t)x+ty\in C$ for each $t\in [0,1].$ Furthermore, we say that a function $f\fun$ is \emph{convex} if the function $f\circ\gamma\col[0,1]\to\exrls$ is convex for every geodesic $\gamma\geo.$

Let $(\hs,d)$ be a Hadamard space and $C\subset \hs$ be a convex closed set. Then for each $x\in \hs$ there exists a unique point $c\in C$ such that
\begin{equation*} d(x,c)=\inf_{y\in C} d(x,y),\end{equation*}
and we denote this point $c$ by $P_C(x).$ The mapping $P_C\col\hs\to C$ is nonexpansive and we call it the \emph{metric projection} onto the set~$C.$ 

Given a function $f\fun,$ we say that a point $z\in \hs$ is a minimizer of~$f$ if 
\begin{equation*} f(z)=\inf_{x\in \hs} f(x).\end{equation*}
The set of all minimizers of $f$ will be denoted $\mi(f).$ A \emph{resolvent} of the function~$f$ is defined by
\begin{equation} \label{eq:res}
J_\lambda(x)\as\argmin_{y\in \hs}\l[f(y)+\frac{1}{2\lambda}d(x,y)^2\r],
\end{equation}
for every $x\in \hs,$ and parameter $\lambda>0.$ If $f$ is convex and lsc, then $J_\lambda\map$ is a~well-defined nonexpansive mapping~\cite[Lemma 2.5]{jost-ch}, and \cite[Lemma 1.12]{mayer}.

Let us now state the following result from~\cite[Lemma 2.2]{gelander}, which then implies the existence of a minimizer of a coercive function in Lemma~\ref{lem:coercive} below.
\begin{lem} \label{lem:gelander}
 Let $(\hs,d)$ be a Hadamard space and $\l(C_n\r)$ be a nonincreasing sequence of bounded closed convex subsets of $\hs.$ Then
\begin{equation*} \bigcap_{n\in\nat} C_n \neq\emptyset.\end{equation*}
\end{lem}
\begin{proof}
 See \cite[Lemma 2.2]{gelander}.
\end{proof}
As a consequence, we obtain the following Lemma~\ref{lem:coercive}. Just recall that a function $f\fun$ is \emph{coercive} if it satisfies $f(x)\to\infty$ whenever $d\l(x,x_0\r)\to\infty,$ for some $x_0\in \hs.$
\begin{lem} \label{lem:coercive}
Let $(\hs,d)$ be a Hadamard space and $f\fun$ be a~coercive convex lsc function. Then $f$ has a minimizer.
\end{lem}
\begin{proof}
We will first observe that $f$ is bounded from below on bounded sets. Let $C\subset \hs$ be bounded, and without loss of generality assume that $C$ is closed convex. If $\inf_C f=-\infty,$ then the sets $S_N=\l\{ x\in C\col f(x)\leq -N\r\}$ for $N\in\nat$ are all nonempty, closed, convex, and bounded. But then Lemma~\ref{lem:gelander} yields a point $z\in \bigcap_{N\in\nat}S_N.$ Clearly $f(z)=-\infty,$ which is not possible.

Since $f$ is bounded from below on bounded sets, it is bounded from below on $\hs,$ by the coercivity assumption. Therefore $\inf_\hs f>-\infty,$ and the sublevel sets
\begin{equation*} C_n\as\l\{x\in \hs\col f(x)\leq \inf_\hs f + \frac1n \r\},\end{equation*}
form a nonincreasing sequence of nonempty, bounded, closed, convex subsets of $\hs.$ Such a family has according to Lemma~\ref{lem:gelander} nonempty intersection and each point in this intersection is obviously a minimizer of~$f.$
\end{proof}

\subsection*{Means.} \label{subsec:means}
Given a~finite set of points $a_1,\dots,a_N\in\hs,$ recall that the (weighted) Fr\'echet mean with positive weights $w_1,\dots,w_N$ satisfying $\sum w_n=1,$ was in~\eqref{eq:mean} defined as
\begin{equation} \label{eq:mean2}
\Xi\as\Xi\l(\ol{w};\ol{a}\r)\as\argmin_{x\in\hs} \sum_{n=1}^N w_n d\l(x,a_n\r)^2,
\end{equation}
where again we denote $\ol{w}\as \l(w_1,\dots,w_N\r)$ and $\ol{a}\as \l(a_1,\dots,a_N\r).$
Some authors alternatively use the name Karcher mean. The existence and uniqueness of the minimizer in the definition is a consequence of nonpositive curvature. It is guaranteed by the following theorem, which is a combination of~\cite[Theorem 3.2.1]{jost2}, \cite[Proposition~4.4]{sturm-conm}, and \cite[Lemma 4.2]{lang-gafa}.
\begin{thm}\label{thm:average}
Let $(\hs,d)$ be a Hadamard space, let $a_1,\dots,a_N\in\hs$ be a finite set of points, and $w_1,\dots,w_N$ be positive weights satisfying $\sum w_n=1.$ Then there exists a unique point $\Xi\in\hs$ defined in~\eqref{eq:mean2}.
Furthermore, this $\Xi$ satisfies the variance inequality
\begin{equation} \label{eq:varineq}
d\l(z,\Xi\r)^2+\sum_{n=1}^N  w_n d\l(\Xi,a_n\r)^2 \leq \sum_{n=1}^N  w_n d\l(z,a_n\r)^2,
\end{equation}
for each $z\in\hs.$
Finally, the function $\Xi(\ol{w};\cdot)$ satisfies
\begin{equation*}d\l(\Xi\l(\ol{w};\ol{a}\r),\Xi\l(\ol{w};\ol{a}'\r) \r) \leq \sum_{n=1}^N w_n d\l(a_n,a_n'\r),\end{equation*}
for every $a_1,\dots,a_N\in\hs,$ and $a_1',\dots,a_N'\in\hs.$
\end{thm}
\begin{proof} We are to show that there exists a unique minimizer of the function
\begin{equation*} \phi\col x\mapsto \sum_{n=1}^N  w_n d\l(x,a_n\r)^2,\quad y\in\hs.\end{equation*}
The function $\phi$ is bounded from below by $0.$ Take a minimizing sequence $\l(y_k\r)\subset\hs,$ that is, a sequence such that $\phi\l(y_k\r)\to\inf\phi.$ The inequality~\eqref{eq:cat} yields that $\l(y_k\r)$ is Cauchy. Indeed, if $y_{kl}$ denotes the midpoint of $y_k$ and $y_l,$ then~\eqref{eq:cat} with $t=\frac12$ gives
\begin{equation*} d\l(y_{kl},a_n\r)^2\leq\frac12 d\l(y_k,a_n\r)^2+\frac12d\l(y_l,a_n\r)^2-\frac14d\l(y_k,y_l\r)^2.\end{equation*}
Multiplying this inequality by $w_n$ and summing from $n=1$ to $N$  easily gives that the sequence $\l(y_k\r)$ is Cauchy. Since $\phi$ is continuous, the sequence $\l(y_k\r)$ converges to a minimizer of $\phi.$ The uniqueness of this minimizer follows again from~\eqref{eq:cat}. It remains to show~\eqref{eq:varineq}.
Employing~\eqref{eq:cat} yields
\begin{align*}
\sum_{n=1}^N w_n d\l(\gamma(t),a_n\r)^2-\sum_{n=1}^N w_n d\l(\Xi,a_n\r)^2 &\leq (1-t)\sum_{n=1}^N w_n\l[d\l(\gamma(0),a_n\r)^2- d\l(\Xi,a_n\r)^2\r] \\ & \quad +t \sum_{n=1}^N w_n \l[d\l(\gamma(1),a_n\r)^2 - d\l(\Xi,a_n\r)^2\r] \\ & \quad-t(1-t)d\l(\gamma(0),\gamma(1)\r)^2,
\end{align*}
for each geodesic $\gamma\geo.$ Setting $\gamma(0)=\Xi$ and $\gamma(1)=z$ gives
\begin{align*}
0 & \leq \sum_{n=1}^N w_n d\l(\gamma(t),a_n\r)^2-\sum_{n=1}^N w_n d\l(\Xi,a_n\r)^2 \\
 & \leq t\l[\sum_{n=1}^N w_n d\l(z,a_n\r)^2-\sum_{n=1}^N w_n d\l(\Xi,a_n\r)^2\r]-t(1-t)d\l(\Xi,z\r)^2
\end{align*}
for each $t\in(0,1).$ Dividing by $t$ and letting $t\to0$ yields~\eqref{eq:varineq}. 

If we denote $\Xi=\Xi\l(\ol{w};\ol{a}\r)$ and $\Xi'=\Xi\l(\ol{w};\ol{a}'\r),$ then Lemma~\ref{lem:resh} yields
\begin{equation*} d\l(a_n,\Xi'\r)^2 + d\l(a_n',\Xi\r)^2 \leq d\l(a_n,\Xi\r)^2 + d\l(a_n',\Xi'\r)^2 +2 d\l(\Xi,\Xi'\r)d\l(a_n,a_n'\r),\end{equation*}
multiplying by $w_n$ and summing up over $n$ from $1$ to $N$ further gives
\begin{align*}
\sum_{n=1}^N w_n \l[ d\l(a_n,\Xi'\r)^2 + d\l(a_n',\Xi\r)^2 \r] &\leq \sum_{n=1}^N w_n \l[ d\l(a_n,\Xi\r)^2 + d\l(a_n',\Xi'\r)^2 \r] \\ & \quad + 2d\l(\Xi,\Xi'\r) \sum_{n=1}^N w_n d\l(a_n,a_n'\r) . 
\end{align*}
By the variance inequality~\eqref{eq:varineq} we have
\begin{align*}
\sum_{n=1}^N w_n \l[ d\l(a_n,\Xi'\r)^2 + d\l(a_n',\Xi\r)^2 \r]  &\geq \sum_{n=1}^N w_n \l[ d\l(a_n,\Xi\r)^2 + d\l(a_n',\Xi'\r)^2 \r] \\ & \quad + 2d\l(\Xi,\Xi'\r)^2. 
\end{align*}
Altogether we obtain
\begin{equation*}d\l(\Xi,\Xi' \r) \leq \sum_{n=1}^N w_n d\l(a_n,a_n'\r),\end{equation*}
which finishes the proof.
\end{proof}

If $a_1,\dots,a_N\in \rls^d,$ then of course
\begin{equation} \label{eq:arithmetic}
 \Xi\l(\ol{w};\ol{a}\r)  = w_1 a_1+\dots+w_N a_N.
\end{equation}

In other words, the Fr\'echet mean coincides with the usual (weighted) arithmetic mean. 

\subsection*{Medians.}
The (weighted) geometric median of a finite set of points $a_1,\dots,a_N\in \hs$ was in~\eqref{eq:median} defined as
\begin{equation} \label{eq:median2}
 \Psi\l(\ol{w};\ol{a}\r)\as\argmin_{x\in \hs} \sum_{n=1}^N w_n d\l(x,a_n\r). 
\end{equation}
Since the function
\begin{equation*} x\mapsto \sum_{n=1}^N w_n d\l(x,a_n\r)\end{equation*}
is convex continuous  and coercive, it has a minimizer due to Lemma~\ref{lem:coercive}. Unlike means, medians are not unique: the set $\Psi\l(\ol{w};\ol{a}\r)$ may contain more than one point in general. As we have already mentioned in the Introduction, a median is an optimal solution to the Fermat-Weber problem in facility location theory.

\subsection*{Supermartingale convergence theorem.} \label{subsec:martingale}
The main results of the present paper rely upon the following form of the supermartingale convergence theorem from \cite[Proposition~4.2]{bert-tsi}, or its deterministic variant, respectively.
\begin{thm} \label{thm:martingale}
Let $\l(\Omega,\cf,\l(\cf_k\r)_{k\in\nato},\mu\r)$ be a filtered probability space. Assume $\l(Y_k\r),\l(Z_k\r)$ and $\l(W_k\r)$ are sequences of nonnegative real-valued random variables defined on $\Omega$ and assume that
\begin{enumerate}
 \item $Y_k,Z_k,W_k$ are~$\cf_k$-measurable for each $k\in\nato,$
 \item $ \expe \l(Y_{k+1}\big|\cf_k\r) \leq Y_k-Z_k+W_k,$ for each $k\in\nato,$ \label{i:mar:ii}
 \item $\sum_k W_k<\infty.$
\end{enumerate}
Then the sequence $\l(Y_k\r)$ converges to a finite random variable $Y$ almost surely, and $\sum_k Z_k<\infty,$ almost surely.
\end{thm}
\begin{proof}
 The proof is now scattered in the literature; see \cite[Proposition~4.2]{bert-tsi}. It is going to appear in a complete and unified form in the forthcoming book~\cite{mybook}.
\end{proof}

A deterministic version of the above theorem will be used in the proof of Theorem~\ref{thm:cyclic}. We include its proof from~\cite[Lemma 3.4]{bert-tsi} for the reader's convenience.
\begin{lem} \label{lem:determinmartingale}
Let $\l(a_k\r),\l(b_k\r)$ and  $\l(c_k\r)$ be sequences of nonnegative real numbers. Assume that
\begin{align}
a_{k+1} & \leq a_k -b_k +c_k, \label{eq:determin} \\
\intertext{for each $k\in\nat,$ and,}
\sum_{k=1}^\infty c_k & < \infty. \nonumber
\end{align}
Then the sequence $\l(a_k\r)$ converges and $\sum_{k=1}^\infty b_k < \infty.$
\end{lem}
\begin{proof}
Fix $l\in\nat.$ Sum~\eqref{eq:determin} over $k\geq l$ and take $\limsup_{k\to\infty}$ to obtain
\begin{equation*} \limsup_{k\to\infty} a_k\leq a_l + \sum_{k=l}^\infty c_k.\end{equation*}
Taking $\liminf_{l\to\infty}$ yields
\begin{equation*}  \limsup_{k\to\infty} a_k \leq  \liminf_{l\to\infty} a_l,\end{equation*}
and hence $\l(a_k\r)$ converges. Now fix $n\in\nat$ and sum~\eqref{eq:determin} from $k=1$ to $k=n,$
\begin{equation*}\sum_{k=1}^n b_k\leq a_1+\sum_{k=1}^n c_k - a_{n+1}.\end{equation*}
Since the last inequality holds for each $n\in\nat,$ we get $\sum_{k=1}^\infty b_k < \infty.$
\end{proof}

\section{The proximal point algorithm} \label{sec:ppa}

The proximal point algorithm (PPA) is a method for finding a minimizer of a~convex lsc function defined on a Euclidean space. Its origins go back to Martinet~\cite{martinet}, Rockafellar~\cite{rocka} and Br\'ezis\&Lions \cite{brezis}. Quite recently, this algorithm was extended into Riemannian manifolds of nonpositive sectional curvature~\cite{sevilla}, and later also into Hadamard spaces~\cite{ppa}. We recall the main result of~\cite{ppa} in Theorem~\ref{thm:ppa} below.

Let $(\hs,d)$ be a Hadamard space and $f\fun$ be a lsc convex function. Assume that $f$ has a minimizer, that is, $\mi(f)\neq\emptyset.$ Given a sequence $\l(\lambda_k\r)$ of positive reals, the \emph{proximal point algorithm} starting at a point $x_0\in \hs$ generates at the $k$-th step, $k\in\nat,$ the point 
\begin{equation} \label{eq:ppa}
 x_k\as\argmin_{y\in \hs}\l[f(y)+\frac1{2\lambda_{k-1}}d\l(y,x_{k-1}\r)^2\r].
\end{equation}
In terms of resolvents, we can equivalently express~\eqref{eq:ppa} as
\begin{equation} \label{eq:ppares}
x_k = J_{\lambda_{k-1}} \l(x_{k-1}\r).
\end{equation}
The convergence of the algorithm was established in~\cite[Theorem 1.4]{ppa}.
\begin{thm} \label{thm:ppa}
Let $(\hs,d)$ be a locally compact Hadamard space and $f\fun$ be a~convex lsc function attaining its minimum on $\hs.$ Then, for an arbitrary starting point $x_0\in \hs$ and a sequence of positive reals $\l(\lambda_k\r)$ such that $\sum_0^\infty\lambda_k=\infty,$ the sequence $(x_k)\subset\hs$ defined by \eqref{eq:ppa} converges to a minimizer of~$f.$
\end{thm}

In the present paper, we consider a function $f\fun$ of the form
\begin{equation} \label{eq:f2}
 f\as\sum_{n=1}^N f_n,
\end{equation}
where $f_n\fun$ are convex lsc, and $N\in\nat.$ In many cases, it is much easier to find the resolvents
\begin{equation}  \label{eq:rescom}
J_\lambda^n(x)\as\argmin_{y\in \hs}\l[f_n(y)+\frac{1}{2\lambda}d(x,y)^2\r]
\end{equation}
of the components $f_n$ than the resolvent of the function $f$ itself. This is true for instance for the median and the mean. Then, instead of applying iteratively the resolvent of $f$ as in~\eqref{eq:ppares}, we will apply the resolvents~\eqref{eq:rescom} of the components $f_n.$ There are essentially two ways of doing that. We either fix an order of the components (that is, a permutation of the numbers $1,\dots,N,$ which without loss of generality may be the identity permutation), and at each cycle we will apply the corresponding resolvents in this fixed order, or alternatively, we will at each step pick a number $r\in\{1,\dots,N\}$ at random, and apply the resolvent of $f_r.$ In either case, we get a sequence converging to a minimizer of $f.$ To be more precise, in the latter situation, we get such a sequence almost surely. For $(\hs,d)$ being the Euclidean space, such results were recently obtained by D.~Bertsekas~\cite{bertsekas}, and we follow his proof strategy.

To prove Theorem~\ref{thm:cyclic} and Lemma~\ref{lem:estimrandom}, we will need the following estimate on the function value at a single PPA step.
\begin{lem} \label{lem:estim} Let $h\fun$ be a convex lsc function on a Hadamard space $(\hs,d),$ and let
\begin{equation*}J_\lambda^h (x)\as\argmin_{z\in \hs}\l[h(z)+\frac{1}{2\lambda}d(x,z)^2\r] \end{equation*}
be its resolvent with parameter $\lambda>0.$ Then
\begin{equation*} h\l(J_\lambda^h (x)\r)-h(y)\leq \frac1{2\lambda} d(x,y)^2 -\frac1{2\lambda} d\l(J_\lambda^h (x),y\r)^2, \end{equation*}
for every $x,y\in \hs.$ 
\end{lem}
\begin{proof}
Choose $x,y\in \hs.$ From the definition of $J_\lambda^h (x)$ we have
\begin{equation*} h\l(J_\lambda^h (x)\r)+\frac1{2\lambda} d(J_\lambda^h (x),x)^2\leq  h(p)+\frac1{2\lambda} d(p,x)^2,\end{equation*}
for each $p\in \hs.$ In particular, let $t\in[0,1)$ and $p_t=(1-t)y+tJ_\lambda^h (x),$ then
\begin{equation*}\frac1{2\lambda} d\l(J_\lambda^h (x),x\r)^2-\frac1{2\lambda} d\l(p_t,x\r)^2\leq h(p_t)-h\l(J_\lambda^h (x)\r) .\end{equation*}
Applying \eqref{eq:cat} to the above inequality gives
\begin{align*}
(1-t)\l[ h(y)-h\l(J_\lambda^h (x)\r) \r]\geq & -\frac{1-t}{2\lambda} d\l(y,x\r)^2 \\ & + \frac{1-t}{2\lambda} d\l(J_\lambda^h (x),x\r)^2 \\ &+\frac{t(1-t)}{2\lambda} d\l(J_\lambda^h (x),y\r)^2,
\end{align*}
or, after taking into account that $t\neq1,$
\begin{equation*}h\l(J_\lambda^h (x)\r)-h(y) \leq \frac1{2\lambda} d\l(y,x\r)^2-\frac1{2\lambda} d\l(J_\lambda^h (x),x\r)^2 -\frac{t}{2\lambda} d\l(J_\lambda^h (x),y\r)^2.\end{equation*}
Passing to the limit $t\to1,$ we conclude that
\begin{equation*}h\l(J_\lambda^h (x)\r)-h(y) \leq \frac1{2\lambda} d\l(y,x\r)^2-\frac1{2\lambda} d\l(J_\lambda^h (x),x\r)^2 -\frac1{2\lambda} d\l(J_\lambda^h (x),y\r)^2,\end{equation*}
which (after neglecting the middle term on the right hand side) finishes the proof.
\end{proof}

\subsection*{Cyclic order version.} 
We will now prove the first main result, namely, that the proximal point algorithm with cyclic order of applying the marginal resolvent gives a sequence which converges to a minimizer. Let us first precisely define the procedure.
\begin{defi} \label{def:cyclic}
Consider a function $f$ of the form~\eqref{eq:f2}. Let $\l(\lambda_k\r)$ be a sequence of positive reals satisfying
\begin{equation} \label{eq:stepsize}
\sum_{k=0}^\infty \lambda_k=\infty,\qquad\text{and}\qquad \sum_{k=0}^\infty \lambda_k^2<\infty.
\end{equation}
Let $x_0\in \hs$ be an arbitrary starting point. For each $k\in\nato$ we set
\begin{align*}
x_{kN+1} & \as J_{\lambda_k}^1\l(x_{kN}\r),\\ x_{kN+2} & \as J_{\lambda_k}^2\l(x_{kN+1}\r),\\ \vdots \\ x_{kN+N} &\as J_{\lambda_k}^N\l(x_{kN+N-1}\r),
\end{align*}
where the resolvents are defined by~\eqref{eq:rescom} above.
\end{defi}
Note that the step size parameter $\lambda_k$ is constant throughout each cycle. The convergence of the above algorithm is assured by the following theorem. The assumption~\eqref{eq:li} will be commented on later in Remark~\ref{rem:lips}.
\begin{thm}[Cyclic order version of the PPA] \label{thm:cyclic} 
Let $(\hs,d)$ be a locally compact Hadamard space, and $f\fun$ be of the form~\eqref{eq:f2} with $\mi(f)\neq\emptyset.$ Given a starting point $x_0\in \hs,$ let $\l(x_j\r)$ be the sequence defined in Definition~\ref{def:cyclic}. Assume there exists $L>0$ such that
\begin{subequations} \label{eq:li}
\begin{align} \label{eq:lipscycl}
 f_n\l(x_{kN}\r)-f_n\l(x_{kN+n}\r) &\leq L d\l(x_{kN}, x_{kN+n}\r), \\
 f_n\l(x_{kN+n-1}\r)-f_n\l(x_{kN+n}\r) &\leq L d\l(x_{kN+n-1}, x_{kN+n}\r),  \label{eq:lipscycl2}
\end{align}
\end{subequations}
for every $k\in\nato,$ and $n=1,\dots,N.$ Then $\l(x_j\r)$ converges to a minimizer of $f.$
\end{thm}
\begin{proof}
We divide the proof into two steps. Step 1:  We claim that
\begin{equation} \label{eq:estimcyclic}
d\l(x_{kN+N},y\r)^2\leq d\l(x_{kN},y\r)^2 -2\lambda_k\l[f\l(x_{kN}\r)-f(y)\r] + 2\lambda_k^2 L^2 N(N+1),
\end{equation}
for each $y\in \hs.$ Indeed, apply Lemma~\ref{lem:estim} with $h=f_n$ and $x=x_{kN+n-1}$ to obtain
\begin{equation*} d\l(x_{kN+n},y\r)^2\leq d\l(x_{kN+n-1},y\r)^2 -2\lambda_k\l[f_n\l(x_{kN+n}\r)-f_n(y)\r], \end{equation*}
for every $y\in \hs,$ and $n=1,\dots,N.$ By summing up we obtain
\begin{align*}
d\l(x_{kN+N},y\r)^2 & \leq d\l(x_{kN},y\r)^2 -2\lambda_k\sum_{n=1}^N \l[f_n\l(x_{kN+n}\r)-f_n(y)\r], \\  & = d\l(x_{kN},y\r)^2 -2\lambda_k\l[f\l(x_{kN}\r)-f(y)\r] \\ & \quad+2\lambda_k\sum_{n=1}^N \l[f_n\l(x_{kN}\r)-f_n\l(x_{kN+n}\r)\r].
\end{align*}
By assumption~\eqref{eq:lipscycl}, we have
\begin{equation*} f_n\l(x_{kN}\r)-f_n\l(x_{kN+n}\r)\leq L d\l(x_{kN}, x_{kN+n}\r)  ,\end{equation*}
where the right hand side can be further estimated as
\begin{equation*} d\l(x_{kN}, x_{kN+n}\r)\leq d\l(x_{kN}, x_{kN+1}\r) +\dots + d\l(x_{kN+n-1}, x_{kN+n}\r).\end{equation*}
By the definition of the algorithm we have
\begin{equation*}f_m\l(x_{kN+m}\r)+\frac1{2\lambda_k}d\l(x_{kN+m-1},x_{kN+m}\r)^2  \leq f_m\l(x_{kN+m-1}\r) ,\end{equation*}
for every $m=1,\dots,N,$ which then gives
\begin{align} 
d\l(x_{kN+m-1},x_{kN+m}\r) &\leq 2\lambda_k \frac{f_m\l(x_{kN+m-1}\r)-f_m\l(x_{kN+m}\r)}{d\l(x_{kN+m-1},x_{kN+m}\r)} \nonumber\\
& \leq 2 \lambda_k L, \label{eq:mstep}
\end{align}
where we employed assumption~\eqref{eq:lipscycl2}. Hence,
\begin{equation*} f_n\l(x_{kN}\r)-f_n\l(x_{kN+n}\r)\leq 2\lambda_k L^2 n,\end{equation*}
and finally,
\begin{equation*} d\l(x_{kN+N},y\r)^2\leq d\l(x_{kN},y\r)^2 -2\lambda_k\l[f\l(x_{kN}\r)-f(y)\r] + 2\lambda_k^2 L^2 N(N+1),\end{equation*}
which finishes the proof of~\eqref{eq:estimcyclic}.

Step 2: Let now $z\in\mi(f),$ and apply~\eqref{eq:estimcyclic} with $y=z.$ Then
\begin{equation*} d\l(x_{kN+N},z\r)^2\leq d\l(x_{kN},z\r)^2 -2\lambda_k\l[f\l(x_{kN}\r)-f(z)\r] + 2\lambda_k^2 L^2 N(N+1),\end{equation*}
which according to Lemma~\ref{lem:determinmartingale} implies that the sequence
\begin{equation*}\l(d\l(x_{kN},z\r)\r)_{k\in\nato}\end{equation*}
converges, (and in particular, the sequence $\l(x_{kN}\r)$ is bounded), and
\begin{equation} \label{eq:summable}
\sum_{k=0}^\infty \lambda_k\l[f\l(x_{kN}\r)-f(z)\r] <\infty.
\end{equation}
From~\eqref{eq:summable} we immediately obtain that there exists a subsequence $\l(x_{k_lN}\r)$ of $\l(x_{kN}\r)$ for which
\begin{equation*}f\l(x_{k_l N}\r)\to f(z),\quad \text{as } l\to\infty.\end{equation*}
Since the sequence $\l(x_{k_lN}\r)$ is bounded, it has a subsequence which converges to a~point $\hat{z}\in \hs.$ By the lower semicontinuity of~$f$ we obtain $\hat{z}\in\mi(f).$ Then we know that
\begin{equation*}\l(d\l(x_{kN},\hat{z}\r)\r)_{k\in\nato}\end{equation*}
converges, and also that it converges to $0,$ since a subsequence of $\l(x_{kN}\r)$ converges to~$\hat{z}.$

By virtue of \eqref{eq:mstep}, we obtain
\begin{equation*}\lim_{k\to\infty} x_{kN+n}=\hat{z},\end{equation*}
for every $n=1,\dots,N.$ Hence the whole sequence $\l(x_j\r)$ converges to $\hat{z}$ and the proof is complete.
\end{proof}

\subsection*{Random order version.}
Instead of applying the marginal resolvents in a~cyclic order, one can at each step select a number from $\{1,\dots,N\}$ at random and use the corresponding resolvent. Next we prove that the resulting sequence converges to a~minimizer of the function $f,$ too.
\begin{defi} \label{def:random}
Let $f$ and $\l(\lambda_k\r)$ be as in Definition~\ref{def:cyclic}. Let $\l(r_k\r)$ be a sequence of random variables which attain values from $\{1,\dots,N\}$ according to the uniform distribution, independently of previous steps. For every $k\in\nato,$ define
\begin{equation} \label{eq:random}
x_{k+1}\as J_{\lambda_k}^{r_k}\l(x_{k}\r),
\end{equation}
with a starting point $x_0\in \hs.$ Finally, denote $x_{k+1}^n$ the result of the iteration with~$x_k$ if $r_k=n.$ Here we of course consider the underlying probability space $\Omega\as\{1,\dots,N\}^\nato$ to be equipped with the product of the uniform probability measure on $\{1,\dots,N\}.$
\end{defi}

The following Lemma~\ref{lem:estimrandom} shows an (almost) supermartingale property required by Theorem~\ref{thm:martingale}\eqref{i:mar:ii}. Again, the assumption~\eqref{eq:lipsrandom} will be commented on in Remark~\ref{rem:lips}.
\begin{lem} \label{lem:estimrandom}
Let $(\hs,d)$ be a Hadamard space and $f$ be of the form~\eqref{eq:f2}. Given a starting point $x_0\in \hs,$ let $\l(x_k\r)$ be the sequence defined in Definition~\ref{def:random}. Assume there exists $L>0$ such that
\begin{equation} \label{eq:lipsrandom}
 f_n\l(x_k\r)-f_n\l(x_{k+1}^n\r)\leq L d\l( x_k,x_{k+1}^n\r),
\end{equation}
for every $k\in\nato$ and $n=1,\dots,N.$ If we denote $\cf_k\as\sigma\l(x_0,\dots,x_k\r),$ then
\begin{equation*} \expe\l[ d\l(x_{k+1},y\r)^2\big| \cf_k\r] \leq d\l(x_k,y\r)^2-\frac{2\lambda_k}{N}\l[f\l(x_k\r)-f(y)\r]+4\lambda_k^2 L^2,\end{equation*}
almost surely, for each $y\in \hs.$
\end{lem}
\begin{proof}
By Lemma~\ref{lem:estim} we have
\begin{equation*} d\l(x_{k+1},y\r)^2\leq d\l(x_k,y\r)^2 - 2\lambda_k\l[ f_{r_k}\l(x_{k+1}\r)-f_{r_k}\l(y\r)\r].\end{equation*}
Taking the conditional expectation with respect to $\cf_k$ gives
\begin{equation*} \expe\l[ d\l(x_{k+1},y\r)^2\big| \cf_k\r] \leq d\l(x_k,y\r)^2 - 2\lambda_k\expe\l[ f_{r_k}\l(x_{k+1}\r)-f_{r_k}\l(y\r)\big| \cf_k\r].\end{equation*}
If we denote $x_{k+1}^n$ the result of the iteration with $x_{k}$ when $r_{k}=n,$ we get
\begin{align*}
\expe\l[ d\l(x_{k+1},y\r)^2\big| \cf_k\r] & \leq d\l(x_k,y\r)^2 - \frac{2\lambda_k}N\sum_{n=1}^N\l[ f_n\l(x_{k+1}^n\r)-f_n\l(y\r)\r] \\ & = d\l(x_k,y\r)^2 - \frac{2\lambda_k}N\l[ f\l(x_k\r)-f(y)\r]  \\ &\quad + \frac{2\lambda_k}N\sum_{n=1}^N\l[ f_n\l(x_k\r)-f_n\l(x_{k+1}^n\r)\r] .
\end{align*}
By the assumption~\eqref{eq:lipsrandom} we have
\begin{equation*}\sum_{n=1}^N\l[ f_n\l(x_k\r)-f_n\l(x_{k+1}^n\r)\r]\leq L \sum_{n=1}^N d\l( x_k,x_{k+1}^n\r)\leq 2L^2\lambda_k N,\end{equation*}
since
\begin{equation*} d\l( x_k,x_{k+1}^n\r)\leq 2\lambda_k \frac{f_n\l(x_k\r)-f_n\l(x_{k+1}^n\r)}{d\l( x_k,x_{k+1}^n\r)}\leq 2\lambda_k L.\end{equation*}
We hence finally obtain
\begin{equation*}\expe\l[ d\l(x_{k+1},y\r)^2\big| \cf_k\r] \leq d\l(x_k,y\r)^2 -\frac{2\lambda_k}{N}\l[f\l(x_k\r)-f(y)\r]+4\lambda_k^2 L^2,\end{equation*}
which finishes the proof.
\end{proof}

We now get to the second convergence theorem.
\begin{thm}[Random order version of the PPA] \label{thm:random}
Let $(\hs,d)$ be a locally compact Hadamard space and $f$ be of the form~\eqref{eq:f2} with $\mi(f)\neq\emptyset.$ Assume that the Lipschitz condition~\eqref{eq:lipsrandom} holds true. Then, given a starting point $x_0\in \hs,$ the sequence $\l(x_k\r)$ defined in Definition~\ref{def:random} converges to a minimizer of $f$ almost surely.
\end{thm}
\begin{proof}
Since $\mi(f)$ is a locally compact Hadamard space, its closed balls are compact by the Hopf-Rinow theorem \cite[p.~35]{bh} and consequently it is separable. We can thus choose a countable dense subset $\l(v_i\r)$ of $\mi(f).$ For each $i\in\nat$ apply Lemma~\ref{lem:estimrandom} with $y=v_i$ to obtain
\begin{equation*} \expe\l[ d\l(x_{k+1}(\omega),v_i\r)^2\big| \cf_k\r] \leq d\l(x_k(\omega),v_i\r)^2-\frac{2\lambda_k}{N}\l[f\l(x_k(\omega)\r)-f\l(v_i\r)\r]+4\lambda_k^2 L^2,
\end{equation*}
for every $\omega$ from a full measure set $\Omega_{v_i}\subset\Omega.$ Theorem~\ref{thm:martingale} immediately gives that $d\l(v_i,x_k(\omega)\r)$ converges, and
\begin{equation*} \sum_{k=0}^\infty \lambda_k\l[f\l(x_k(\omega)\r)-\inf f\r]<\infty,\end{equation*}
for every $\omega\in\Omega_{v_i}.$ Next denote
\begin{equation*}\Omega_\infty\as \bigcap_{i\in\nat} \Omega_{v_i},\end{equation*}
which is by countable subadditivity again a set of full measure. The last inequality yields that for $\omega\in\Omega_\infty,$ we have $\liminf_{k\to\infty} f\l(x_k(\omega)\r)= \inf f,$ and since $\l(x_k(\omega)\r)$ is bounded, it has a cluster point $x(\omega)\in \hs.$ By the lower semicontinuity of $f$ we may assume that $x(\omega)\in\mi(f).$

For each $\eps>0$ there exists $v_{i(\eps)}\in\l(v_i\r)$ such that $d\l(x(\omega),v_{i(\eps)}\r)<\eps.$ Because the sequence $d\l(x_k(\omega),v_{i(\eps)}\r)$ converges and $x(\omega)$ is a cluster point of $x_k(\omega),$ we have
\begin{equation*}\lim_{k\to\infty} d\l(x_k(\omega),v_{i(\eps)}\r)<\eps.\end{equation*}
This yields $x_k(\omega)\to x(\omega).$ We obtain that $x_k$ converges to a minimizer almost surely. This finishes the proof.
\end{proof}

\begin{rem} \label{rem:lips}
 The assumptions~\eqref{eq:li} in Theorem~\ref{thm:cyclic}, and~\eqref{eq:lipsrandom} in Theorem~\ref{thm:random} are satisfied, for instance, if
\begin{enumerate}
 \item the functions $f_n$ are Lipschitz on $\hs$ with constant $L,$ or \label{i:lips:i}
 \item the function $f$ is of the form~\eqref{eq:weighted}.      \label{i:lips:ii}
\end{enumerate}
In particular, for both the mean~\eqref{eq:mean2}, and the median~\eqref{eq:median2}. While the Lipschitz condition in~\eqref{i:lips:i} is clear, we note that in case of~\eqref{i:lips:ii}, the PPA sequences are bounded because they lie in the closed convex hull of $\l\{x_0,a_1,\dots,a_N\r\}$ and the functions $f_n\as w_n d\l(\cdot,a_n\r)^p$ are locally Lipschitz.
\end{rem}
To summarize (the most important case of) the results in this section, we state the following corollary.
\begin{cor} \label{cor:summarize}
 Let $(\hs,d)$ be a locally compact Hadamard space and $f$ be of the form~\eqref{eq:f2} with $\mi(f)\neq\emptyset.$ Assume that (at least) one of the following conditions is satisfied:
\begin{enumerate}
 \item the functions $f_n$ are Lipschitz on $\hs$ with constant $L,$ or
 \item the function $f$ is of the form~\eqref{eq:weighted}.
\end{enumerate}
Let $x_0\in \hs.$ Then:
\begin{enumerate}
 \item The sequence defined in Definition~\ref{def:cyclic} converges to a minimizer of $f.$
 \item The sequence defined in Definition~\ref{def:random} converges to a minimizer of $f$ almost surely.
\end{enumerate}
\end{cor}
\begin{proof}
 The proof follows immediately from Remark~\ref{rem:lips}.
\end{proof}
\begin{rem}
 It is easy to observe that if the function $f$ is of the form~\eqref{eq:weighted}, the approximating sequences stay in the closed convex hull of the points $a_1,\dots,a_N.$ If we knew that this closed convex hull is a compact set, we could drop the assumption that $\hs$ is locally compact. Unfortunately, it is not known in a general Hadamard space whether the closed convex hull of a finite set is compact; see \cite[Section 4]{jost94} and also \cite{kopecka-reich}.
\end{rem}

\section{Computing medians and means} \label{sec:mm}

The algorithms from Definitions~\ref{def:cyclic} and~\ref{def:random} can be directly applied to compute means and medians in locally compact Hadamard spaces. We will now give an explicit description of these two special cases. It is interesting to observe how the highly multidimensional optimization problem of minimizing the function~\eqref{eq:multidim} is converted to a sequence of one-dimensional optimization problems of minimizing the function in~\eqref{eq:onedim}; and likewise in the case of medians.

\subsection*{Algorithms for computing means}
Given positive weights $w_1,\dots,w_N$ with $\sum w_n=1$ and points $a_1,\dots,a_N\in \hs,$ we wish to minimize the function
\begin{equation} \label{eq:multidim}
 f(x)\as\sum_{n=1}^N w_n d\l(x,a_n\r)^2, \qquad x\in \hs.
\end{equation}
The existence and uniqueness of a minimizer is assured by Theorem~\ref{thm:average}. Furthermore, the function $f$ is of the form~\eqref{eq:weighted} and according to Corollary~\ref{cor:summarize} satisfies both~\eqref{eq:li} and~\eqref{eq:lipsrandom}. We can therefore employ proximal point algorithms with $f_n=w_n d\l(\cdot,a_n\r)^2,$ for $n=1,\dots,N.$ Let us first consider the cyclic order version from Definition~\ref{def:cyclic}. Let $\l(\lambda_k\r)$ be a sequence of positive reals satisfying~\eqref{eq:stepsize}. We start at some point $x_0\in \hs,$ and for each $k\in\nato$ we set
\begin{align*}
x_{kN+1} & \as J_{\lambda_k}^1\l(x_{kN}\r),\\ x_{kN+2} & \as J_{\lambda_k}^2\l(x_{kN+1}\r),\\ \vdots \\ x_{kN+N} &\as J_{\lambda_k}^N\l(x_{kN+N-1}\r),
\end{align*}
where $J_{\lambda_k}^n$ is now the resolvent of the function $f_n= w_n d\l(\cdot,a_n\r)^2,$ for $n=1,\dots,N.$ It is easy to find these resolvents explicitly. Indeed, fix $k\in\nato$ and $n=1,\dots,N.$ Then $x_{kN+n}$ is the unique minimizer of the function
\begin{equation} \label{eq:onedim}
w_n d\l(\cdot,a_n\r)^2 +\frac1{2\lambda_k} d\l(\cdot,x_{kN+n-1}\r)^2,
\end{equation}
and it is obvious that such a minimizer lies on the geodesic $\l[x_{kN+n-1},a_n\r],$ that is, 
\begin{equation*}x_{kN+n}=\l(1-t_k^n\r) x_{kN+n-1}+ t_k^n a_n,\end{equation*}
for some $t_k^n\in[0,1].$ By an elementary calculation we get
\begin{equation} \label{eq:tkcoef}
t_k^n=\frac{2\lambda_k w_n}{1+2\lambda_k w_n}. 
\end{equation}
The above algorithm then reads:
\begin{alg}[Computing mean, cyclic order version] \label{alg:meancyc}
Given $x_0\in \hs$ and $\l(\lambda_k\r)$ satisfying~\eqref{eq:stepsize} we set
 \begin{align*}
x_{kN+1} & \as \frac{1}{1+2\lambda_kw_1}x_{kN} + \frac{2\lambda_kw_1}{1+2\lambda_kw_1} a_1,\\ x_{kN+2} & \as \frac{1}{1+2\lambda_kw_2} x_{kN+1}+\frac{2\lambda_kw_2}{1+2\lambda_kw_2} a_2,\\ \vdots \\ x_{kN+N} &\as \frac{1}{1+2\lambda_kw_N}x_{kN+N-1}+\frac{2\lambda_kw_N}{1+2\lambda_kw_N}a_N,
\end{align*}
for each $k\in\nato$ and $n=1,\dots,N.$
\end{alg}
The convergence of the sequence $\l(x_j\r)$ produced by Algorithm~\ref{alg:meancyc} to the weighted mean of the points $a_1,\dots,a_N$ follows by Theorem~\ref{thm:cyclic} above. Note that if the weights are uniform, that is, $w_n=\frac1N$ for each $n=1,\dots,N,$ then the coefficients~$t_k^n$ are independent of $n.$

We will now turn to the randomized version from Definition~\ref{def:random}. By a similar process as above we obtain the following algorithm.
\begin{alg}[Computing mean, random order version] \label{alg:meanran}
Let $x_0\in \hs$ be a starting point and $\l(\lambda_k\r)$ satisfy~\eqref{eq:stepsize}. At each step $k\in\nato,$ choose randomly $r_k\in\{1,\dots,N\}$ according to the uniform distribution and put
\begin{equation*} x_{k+1}\as  \frac{1}{1+2\lambda_k w_{r_k}}x_k + \frac{2\lambda_k w_{r_k}}{1+2\lambda_k w_{r_k}} a_{r_k}.\end{equation*}
\end{alg}
The convergence of the sequence $\l(x_k\r)$ produced by Algorithm~\ref{alg:meanran} to the weighted mean of the points $a_1,\dots,a_N$ follows by Theorem~\ref{thm:random} above.

\subsection*{Algorithms for computing medians}
Given positive weights $w_1,\dots,w_N$ with $\sum w_n=1$ and points $a_1,\dots,a_N\in \hs,$ we wish to minimize the function
\begin{equation*} f(x)\as\sum_{n=1}^N w_n d\l(x,a_n\r), \qquad x\in \hs.\end{equation*}
The function $f$ is again of the form~\eqref{eq:f2} with $f_n=w_n d\l(\cdot,a_n\r),$ for $n=1,\dots,N.$ It is Lipschitz, and hence satisfies the assumptions~\eqref{eq:li} and~\eqref{eq:lipsrandom}. In the cyclic order version, we start at some point $x_0\in \hs,$ and for each $k\in\nato$ we set
\begin{align*}
x_{kN+1} & =J_{\lambda_k}^1\l(x_{kN}\r),\\ x_{kN+2} & =J_{\lambda_k}^2\l(x_{kN+1}\r),\\ \vdots \\ x_{kN+N} &=J_{\lambda_k}^N\l(x_{kN+N-1}\r),
\end{align*}
where $J_{\lambda_k}^n$ is the resolvent of the function $f_n=w_n d\l(\cdot,a_n\r),$ for $n=1,\dots,N,$ and $\l(\lambda_k\r)$ is a sequence of positive reals satisfying~\eqref{eq:stepsize}. More specifically, if we fix $k\in\nato$ and $n=1,\dots,N,$ then $x_{kN+n}$ is the unique minimizer of the function
\begin{equation*} w_n d\l(\cdot,a_n\r) +\frac1{2\lambda_k} d\l(\cdot,x_{kN+n-1}\r)^2,\end{equation*}
and it is obvious that such a minimizer lies on the geodesic $\l[x_{kN+n-1},a_n\r],$ that is, 
\begin{equation*}x_{kN+n}=\l(1-t_k^n\r) x_{kN+n-1}+ t_k^n a_n,\end{equation*}
for some $t_k^n\in[0,1].$ These coefficients are again easy to determine. We have to however treat the cyclic and the random case separately.
\begin{alg}[Computing median, cyclic order version] \label{alg:mediancyc}
Given $x_0\in \hs$ and $\l(\lambda_k\r)$ satisfying~\eqref{eq:stepsize} we set
\begin{align*}
x_{kN+1} & \as \l(1-t_k^1 \r) x_{kN} + t_k^1 a_1,\\ x_{kN+2} & \as \l(1-t_k^2 \r) x_{kN+1}+t_k^2 a_2,\\ \vdots \\ x_{kN+N} &\as \l(1-t_k^N \r) x_{kN+N-1}+t_k^N a_N,
\end{align*}
with $t_k^n$ defined by
\begin{equation*} 
t_k^n\as \min\l\{1,\frac{\lambda_k w_n}{d\l(a_n,x_{kN+n-1}\r)}\r\},
\end{equation*}
for each $k\in\nato$ and $n=1,\dots,N.$ 
\end{alg}
The convergence of the sequence $\l(x_j\r)$ produced by Algorithm~\ref{alg:mediancyc} to the median of the points $a_1,\dots,a_N$ with the weights $\l(w_1,\dots,w_N\r)$ follows by Theorem~\ref{thm:cyclic} above. Finally, the randomized version can be derived in a similar way.
\begin{alg}[Computing median, random order version] \label{alg:medianran}
Let $x_0\in \hs$ be a~starting point and $\l(\lambda_k\r)$ satisfies~\eqref{eq:stepsize}. At each step $k\in\nato,$ choose randomly $r_k\in\{1,\dots,N\}$ according to the uniform distribution and put
\begin{equation} \label{eq:medran}
 x_{k+1}\as  \l(1-t_k \r)x_k + t_k a_{r_k},
\end{equation}
with $t_k$ defined by
\begin{equation*} 
t_k\as \min\l\{1,\frac{\lambda_k w_{r_k}}{d\l(a_{r_k},x_k\r)}\r\},
\end{equation*}
for each $k\in\nato.$ 
\end{alg}
The convergence of the sequence $\l(x_k\r)$ produced by Algorithm~\ref{alg:medianran} to the median of the points $a_1,\dots,a_N$ follows by Theorem~\ref{thm:random} above. 

\begin{rem} \label{rem:boris}
Let now take a look at a more general situation mentioned already in the Introduction. Let $C_1,\dots,C_N$ be convex closed subsets of our locally compact Hadamard space $(\hs,d)$ and minimize the function~\eqref{eq:boris}, that is,
\begin{equation*}f(x) \as  \sum_{n=1}^N w_n d\l(x;C_n\r),\qquad x\in \hs,\end{equation*}
where $\ol{w}\as \l(w_1,\dots,w_N\r)$ are again positive weights with $\sum w_n=1.$ We have to assume that at least one of the sets $C_1,\dots,C_N$ is bounded in order to fulfill the assumption $\mi(f)\neq\emptyset$ in Theorems~\ref{thm:cyclic} and~\ref{thm:random}. It is also clear that $f$ is convex and $1$-Lipschitz and thus satisfies both~\eqref{eq:li} and~\eqref{eq:lipsrandom}. Let $P_n$ denote the metric projection onto the set $C_n,$ where $n=1,\dots,N.$ We describe the random version of the PPA algorithm only, the cyclic version being completely analogous.

Let $x_0\in \hs$ be a~starting point and $\l(\lambda_k\r)$ satisfies~\eqref{eq:stepsize}. At each step $k\in\nato,$ choose randomly $r_k\in\{1,\dots,N\}$ according to the uniform distribution and put
\begin{equation*} x_{k+1}\as  \l(1-t_k \r)x_k + t_k P_{r_k}\l(x_k\r), \end{equation*}
with $t_k$ defined by
\begin{equation*} t_k\as \min\l\{1,\frac{\lambda_k w_{r_k}}{d\l(P_{r_k}\l(x_k\r),x_k\r)}\r\}, \end{equation*}
for each $k\in\nato.$ Then the sequence~$\l(x_k\r)$ converges to a minimizer of $f$ by Theorem~\ref{thm:random}.
\end{rem}

\section{Computing means via the law of large numbers} \label{sec:lln}

In this last section we give an alternative algorithm for computing the Fr\'echet mean, which is based on the law of large numbers. The advantage of this approach is that we do not require the underlying Hadamard space be locally compact. We shall also compare this algorithm with Algorithm~\ref{alg:meanran}.

Let again $a_1,\dots,a_N\in\hs$ be a finite set of points and $w_1,\dots,w_N$ be positive weights satisfying $\sum w_n=1.$ Denote the probability measure
\begin{equation} \label{eq:distr}
 \pi\as \sum_{n=1}^N w_n\delta_{a_n},
\end{equation}
where $\delta_{a_n}$ stands for the Dirac measure at $a_n.$ Assume that $Y$ is a random variable with values in $\hs$ distributed according to~$\pi.$ Then the variational inequality~\eqref{eq:varineq} can be written as
\begin{equation} \label{eq:varineq2}
d\l(z,\Xi\r)^2+\expe d\l(\Xi,Y\r)^2 \leq \expe d\l(z,Y\r)^2,\quad z\in\hs,
\end{equation}
where the expectation $\expe$ is of course taken with respect to the distribution $\pi.$

Given a sequence of random variables $Y_k$ with values in $\hs,$ we define a sequence~$(S_k)$ of random variables putting $S_1\as Y_1,$ and
\begin{equation} \label{eq:aritmean}
S_{k+1} \as \frac{k}{k+1} S_k+\frac{1}{k+1}Y_{k+1},
\end{equation}
for $i\in\nat.$ The random variables $Y_k,$ and hence also $S_k,$ are defined on some probability space $\Omega,$ but this space $\Omega$ of course plays no role here. The following theorem due to K.-T.~Sturm states a nonlinear version of the law of large numbers. It appeared in a much more general form in~\cite[Theorem 2.6]{sturm}.
\begin{thm}[The law of large numbers]\label{thm:largenumbers}
Let $(\hs,d)$ be a Hadamard space, and $\l(Y_k\r)$ be a sequence of independent random variables $Y_k\col\Omega\to\hs,$ identically distributed according to the distribution $\pi,$ defined in~\eqref{eq:distr}. Then
\begin{equation*} S_k\to \Xi\l(\ol{w};\ol{x}\r),\quad\text{as } k\to\infty,\end{equation*}
where the convergence is pointwise.
\end{thm}
\begin{proof} First denote
\begin{equation*} \xi \as \min_{x\in\hs} \sum_{n=1}^N w_n d\l(x,a_n\r)^2.\end{equation*}
We show by induction on $k\in\nat$ that
\begin{equation} \label{eq:largenum1}
\expe d\l(\Xi,S_k\r)^2\leq\frac1{k}\xi.
\end{equation}
It obviously holds for $k=1$ and we assume it holds for some $k\in\nat.$ We have
\begin{align*}
\expe d\l(\Xi,S_{k+1}\r)^2 & =\expe d\l(\Xi,\frac{k}{k+1}S_k+\frac{1}{k+1}Y_{k+1}\r)^2 ,\\
\intertext{by~\eqref{eq:cat} we get}
 & \leq\frac{k}{k+1}\expe d\l(\Xi,S_k \r)^2 + \frac{1}{k+1}\expe d\l(\Xi,Y_{k+1}\r)^2 - \frac{k}{(k+1)^2}\expe d\l( Y_{k+1},S_k \r)^2, 
\\
\intertext{and applying independence and~\eqref{eq:varineq2} gives}
 & \leq \frac{k}{k+1}\expe d\l(\Xi,S_k\r)^2  + \frac{1}{k+1}\expe d\l(\Xi,Y_{k+1}\r)^2 \\ & \qquad - \frac{k}{(k+1)^2}\expe \l[ d\l(\Xi,S_k \r)^2 + d\l(\Xi,Y_{k+1}\r)^2\r] \\
& = \l(\frac{k}{k+1} \r)^2\expe d\l(\Xi,S_k\r)^2+\frac1{(k+1)^2}\xi\\ &\leq\frac1{k+1}\xi.
\end{align*}
This shows that~\eqref{eq:largenum1} holds, and hence the proof is complete.
\end{proof}

One can rather straightforwardly convert Theorem~\ref{thm:largenumbers} into an approximation algorithm for computing the Fr\'echet mean. Let us now describe such an algorithm. It receives the points $a_1,\dots,a_N$ and weights $w_1,\dots,w_N$ as the input, and at each iteration $k\in\nat$ it produces a new point $s_k\in\hs,$ which is an approximate version of the desired mean $\Xi\as \Xi\l(\ol{w};\ol{x}\r)$ in the sense that $d\l(s_k,\Xi\r)\to0$ as $k\to\infty.$ The sequence is defined as follows. At each step $k\in\nato,$ choose randomly $r_k\in\{1,\dots,N\}$ according to the distribution $\ol{w}=\l(w_1,\dots,w_N\r)$ and put
\begin{equation} \label{eq:llnalg}
 s_{k+1}\as\frac{k}{k+1} s_k+\frac{1}{k+1} a_{r_k}.
\end{equation}
The convergence of this algorithm is guaranteed by Theorem~\ref{thm:largenumbers}.

We shall now compare the algorithm~\eqref{eq:llnalg} with Algorithm~\ref{alg:meanran}. Let us first consider the \emph{unweighted} case, that is, $w_n=\frac1N$ for every $n=1,\dots,N.$ At each iteration $k\in\nato,$ the algorithm~\eqref{eq:llnalg} selects $r_k\in\{1,\dots,N\}$ according to the \emph{uniform} distribution and generates a new point
\begin{equation*}s_{k+1}\as\frac{k}{k+1} s_k+\frac{1}{k+1} a_{r_k}.\end{equation*}
In Algorithm~\ref{alg:meanran}, at each step $k\in\nato$ we randomly choose a number $r_k\in\{1,\dots,N\}$ according to the \emph{uniform} distribution and put
\begin{equation*} x_{k+1}\as \frac{1}{1+2\lambda_k}x_k + \frac{2\lambda_k }{1+2\lambda_k} a_{r_k}.\end{equation*}
Thus Algorithm~\ref{alg:meanran}  produces the same sequence as the algorithm~\eqref{eq:llnalg} provided we set $\lambda_k\as\frac1{2k}$ for each $k\in\nat.$ In other words the algorithm~\eqref{eq:llnalg} is a special case of Algorithm~\ref{alg:meanran}. 

On the other hand as far as \emph{weighted} Fr\'echet means are concerned, there exists a difference between these two algorithms. Indeed, if $\ol{w}\as\l(w_1,\dots,w_N\r)$ are the weights, then the algorithm~\eqref{eq:llnalg} selects $r_k\in\{1,\dots,N\}$ according to the \emph{distribution}~$\ol{w}$ and generates a new point
\begin{equation*}s_{k+1}\as\frac{k}{k+1} s_k+\frac{1}{k+1} a_{r_k},\end{equation*}
that is, with the same coefficients as in the unweighted case. Algorithm~\ref{alg:meanran} in contrast still selects $r_k\in\{1,\dots,N\}$ according to the \emph{uniform} distribution, but the new point is given by
\begin{equation*} x_{k+1}= \frac{1}{1+2\lambda_k w_{r_k}}x_k + \frac{2\lambda_k w_{r_k}}{1+2\lambda_k w_{r_k}} a_{r_k},\end{equation*}
that is, the coefficients now do depend on the weights. In summary, introducing weights effects either the coefficients (Algorithm~\ref{alg:meanran}), or the probability distribution which is used for selecting the points $a_1,\dots,a_N$ (the algorithm~\eqref{eq:llnalg}).

\subsection*{Final remarks} \label{finalremarks}
Notice that all the algorithms presented in this paper require finding a geodesic at each iteration. For instance, in Algorithm~\ref{alg:medianran}, we need to find the geodesic $\l[x_k,a_{r_k}\r]$ at each step $k\in\nato,$ or more precisely, we need to compute the point $x_{k+1}$ which lies on this geodesic. When employing these algorithms in the BHV tree space, we can use the Owen-Provan algorithm (mentioned in the Introduction) to find this point in polynomial time.


\bibliographystyle{siam}
\bibliography{median-mean}

\end{document}